\documentclass[12pt,a4paper]{amsart}
\usepackage{amsmath,amsfonts,amsthm,amscd, mathrsfs,bbm, amsthm,amssymb,multirow, tensor,tikz}

\newtheorem{defn}{Definition}
\newtheorem{thm}{Theorem}
\newtheorem{prop}{Proposition}

\newtheorem{lemma}{Lemma}

\newtheorem{remk}{Remark}
\newtheorem{exam}{Example}

\newcommand{\pp}{\boldsymbol{p}}
\newcommand{\cT}{\mathcal{T}}

\newcommand{\isom}{\operatorname{Isom{}_+}}

\def\sideremark#1{\ifvmode\leavevmode\fi\vadjust{\vbox to0pt{\vss
 \hbox to 0pt{\hskip\hsize\hskip1em
 \vbox{\hsize3cm\tiny\raggedright\pretolerance10000
  \noindent #1\hfill}\hss}\vbox to8pt{\vfil}\vss}}}%
                        

\author{A.\ Rod Gover and Callum Sleigh}

\title[Tractors, BGG complexes, and group  cohomology]{Tractor calculus, BGG complexes, and the cohomology of Kleinian groups}

\address{A.R.G. and C.S.: Department of Mathematics\\
   The University of Auckland\\
   Private Bag 92019\\
   Auckland 1142\\
   New Zealand}

 \email{r.gover@auckland.ac.nz}
 \email{c.sleigh@auckland.ac.nz}


\begin{document}
\maketitle

\begin{abstract} 
For a compact, oriented, hyperbolic $n$-manifold $(M,g)$, realised as
$M= \Gamma \backslash \mathbb{H}^{n}$ where $\Gamma$ is a torsion-free
cocompact subgroup of $SO(n,1)$, we establish and study a relationship
between differential geometric cohomology on $M$ and algebraic
invariants of the group $\Gamma$. In particular for $\mathbb{F}$ an
irreducible $SO(n,1)$-module, we show that the group cohomology with
coefficients $H^{\bullet}(\Gamma, \mathbb{F})$ is isomorphic to the
cohomology of an appropriate projective BGG complex on $M$.  This
yields the geometric interpretation that $H^{\bullet}(\Gamma,
\mathbb{F})$ parameterises solutions to certain distinguished natural
PDEs of Riemannian geometry, modulo the range of suitable differential coboundary
operators. Viewed in another direction, the construction shows one way that
non-trivial cohomology can arise in a BGG complex, and sheds
considerable light on its geometric meaning.  We also use the tools
developed to give a new proof that $H^{1} (\Gamma, S_{0}^{k}
\mathbb{R}^{n+1}) \neq 0$ whenever $M$ contains a compact, orientable,
totally geodesic hypersurface. All constructions use another result
that we establish, namely that the canonical flat connection on a
hyperbolic manifold coincides with the tractor connection of
projective differential geometry.
\end{abstract}
\maketitle

\thanks{ARG gratefully acknowledges support from the Royal
  Society of New Zealand via Marsden Grant 13-UOA-018 }

\section{Introduction}

Hyperbolic manifolds and their fundamental groups are objects that lie
at the intersection of many different areas of mathematics.  Indeed, their study has exposed deep interactions between analysis, group
theory, Riemannian geometry, and topology, see for example
\cite{Gromov,Milnor,Thurston,BenedettiPetronio}.  In this article we
show that there is considerable gain in introducing another tool to study hyperbolic manifolds; namely, \emph{projective differential geometry}.
We show that many 
standard objects in the area may be effectively interpreted in this
framework.  Critically, this perspective enables us to apply powerful
tools such as the Cartan-tractor calculus and BGG theory
\cite{BailEastGov,CaldDiemer,CapGover,CSS}; these provide a route
for the geometric interpretation of certain algebraic invariants of hyperbolic
geometries.

Projective differential geometry is based around geodesic structure.
Two affine connections, $\nabla$ and $\nabla'$, are said to be
projectively related if they share the same geodesics as
unparameterised curves. This happens trivially if connections differ
only by torsion, so we take a {\em projective geometry} to mean a
manifold equipped with an equivalence class $\pp=[\nabla]$ of
torsion-free affine connections which have the same geodesics up to
parametrisation. In Riemannian geometry the metric canonically
determines a distinguished affine connection, the
Levi-Civita connection, and this provides a basic object for
invariantly treating local geometric analysis. On projective
geometries there is no such distinguished connection on the tangent
bundle.  However there is a canonical connection on a related
higher rank bundle, called the \emph{tractor connection}, and for
projective geometry this is the basic tool for invariant local
analysis.

Since hyperbolic manifolds have a Riemannian metric and Levi-Civita
connection available, our reasons for exploiting the tractor
connection are more subtle.  The first point is that on a hyperbolic
manifold the tractor connection is flat; indeed a hyperbolic manifold
may be characterised as a flat projective manifold equipped with a
certain Cartan holonomy reduction in the sense of
\cite{CapGoverHammerl,CGH}. In the hyperbolic geometry literature it
is well known that the hyperbolic metric gives rise to another canonical flat connection, determined by the 
faithful representation of the fundamental group into the group of hyperbolic isometries. We show that this agrees with the tractor
connection, and the identification of these
connections is used to give a rich interplay between results in
projective differential geometry and results for (higher dimensional)
Kleinian groups.  In particular, we use the projective tractor calculus
and the related BGG theory (as discussed below) to study the
cohomology, with coefficients, of the fundamental group of a
hyperbolic manifold.

The work here also has motivation from another direction; namely, the
general role of projective geometry in pseudo-Riemannian geometry.
Clearly any pseudo-Riemannian geometry determines a projective
structure through its Levi-Civita connection. However a deeper role of
projective geometry in the metric setting has emerged \cite{Mikes,Sinjukov}, with already some striking developments
and applications \cite{BDE,KM,matv}. A particularly promising aspect is a
strong link with certain overdetermined partial differential equations
and their close connections to the invariant calculus for projective
geometry \cite{CGprojectiveCom,CapGoverHammerl,EastMat}. The current
work provides a new direction in this programme.

Since the work here brings together different areas of mathematics we
have attempted to present the material in an elementary way and, to
the extent possible, make the treatment self-contained.
 
Before discussing our results in more detail, let us point out that
there are two directions in which the work here should
generalise. First compact hyperbolic manifolds form our main focus
here, but much of the development applies equally to the case of
infinite volume hyperbolic manifolds; this should be especially
interesting because of the relation to projective compactification
defined in \cite{CGprojectiveCom}. Second, many of the constructions in this article will generalise without
difficulty to many classes of manifolds that arise by discrete quotients
of homogeneous spaces.  We view our treatment here as a template for the first
aspects of these wider theories.

\subsection{Main results}\label{mr}

We work on a compact, oriented, hyperbolic $n$-manifold $(M,g)$, $n\geq
2$, realised as $M= \Gamma \backslash \mathbb{H}^{n}$ where $\Gamma$
is a torsion-free cocompact subgroup of $SO(n,1)$. As mentioned above, it is well known 
that for any $SO(n,1)$-module $\mathbb{F}$ there is canonical vector
bundle and connection $(F,\nabla^{dev})$ (with typical fibre
$\mathbb{F}$) with the property that the holonomy of $\nabla^{dev}$
recovers the representation of $\Gamma$ on $\mathbb{F}$ (see, for example, Part $1$ of \cite{MatsushimaMurakami} or Section
\ref{locvsTr} below). On the other hand, since the metric $g$
determines a projective structure, for any $SL(n+1,\mathbb{R})$-module
$\mathbb{E}$ there is a corresponding projective tractor bundle
$\mathcal{T}_{\mathbb{E}}$ on $M$ with projective tractor connection
$\nabla^{\mathcal{T}_{\mathbb{E}}}$. The first main result is Theorem
\ref{MainIsom} which states that these connections agree in the case
that $\mathbb{F}$ is the restriction of $\mathbb{E}$ to $SO(n,1)
\subset SL(n+1,\mathbb{R})$. We paraphrase:

\smallskip

\noindent {\bf Theorem.} 
{\it There is a canonical isomorphism of vector bundles and \textup{(}flat\textup{)} connections,}
\[
(F, \nabla^{dev}) \cong (\mathcal{T}_{\mathbb{E}}, \nabla^{\mathcal{T}_{\mathbb{E}}}).
\]

\medskip

Projective geometries fit into a large class of structures known as 
parabolic geometries. Within this class there are fundamental
differential sequences known as BGG sequences.  Building on work of
Baston \cite{Baston1, Baston2} and others, the general theory of these was
developed by \v{C}ap, Slov\`{a}k, Sou\v{c}ek \cite{CSS} and
Calderbank, Diemer \cite{CaldDiemer}; they are intimately related to
the algebraic Bernstein-Gelfand-Gelfand resolutions of
\cite{BGG,Lepowsky}.  Part of the importance of these sequences arises
from the fact that they form complexes when the geometry is flat in
the sense of its Cartan connection.  While there has been considerable
study of the local structure of these sequences, see e.g.\ \cite{CS2}, and
also geometric interpretation of the solution space for the first
operators \cite{Armstrong,CapGoverHammerl,CGH,Leitner}
there has been
little known about the higher cohomology of BGG complexes. Another
key motivation of this article is to make some first inroads to 
address this gap. 

From the general theory it follows that on a hyperbolic manifold there is, for each  irreducible representation  $\mathbb{E}$
of $SL(n+1)$, a projective BGG complex that takes the form,
\newcommand{\cB}{\mathcal{B}}
$$
  0 \longrightarrow \cB^0 \stackrel{D_0}{\longrightarrow} \cB^1\stackrel{D_1}{\longrightarrow} \cdots \stackrel{D_{n-1}}{\longrightarrow} \cB^n 
\longrightarrow  0,
$$
where the $\cB^i$ are irreducible (weighted) tensor bundles.  Each of
these complexes may be derived from a twisting of the de Rham complex by
the corresponding (flat) tractor connection, and computes the same
cohomology as the tractor twisted de Rham complex.  Using this fact, and
the isomorphism from Theorem \ref{MainIsom} discussed above brings
us to the next main result, Theorem \ref{BGGIsomGroup}, which may be
summarised as: 

\smallskip

\noindent {\bf Theorem.}  {\it The cohomology of the projective
  Bernstein-Gelfand-Gelfand complex associated to $\mathbb{E}$ is
  isomorphic to the group cohomology $H^{\bullet}(\Gamma,
  \mathbb{F})$.}

\medskip

The basic BGG theory that we need is reviewed in Section \ref{BGGsec}
with examples in Section \ref{BGGex}.  Theorem \ref{BGGCohom} and
Theorem \ref{BGGIsomGroup} provide a nontrivial geometric
interpretation of the group cohomology of $\Gamma$, with 
coefficients in any $SO(n,1)$-module $\mathbb{F}$. 
Apart from certain isolated cases
a truly geometric interpretation of the group cohomology of $\Gamma$
with coefficients has been unknown until now, to the best of our
knowledge.

Finally, in section \ref{MainNonVanishing}, we use tractor-theoretic methods to give a proof of the following non-vanishing theorem. Here, and throughout, we will use the notation $S^{k}_{0} \mathbb{R}^{n+1}$ 
for the $SO(n,1)$-module formed by taking the trace-free part of the $k$th symmetric power of the defining representation. The trace comes from the signature $(n,1)$ inner product on $\mathbb{R}^{n+1}$.

\noindent {\bf Theorem.}
{\it If $\Gamma \backslash \mathbb{H}^{n}$ is a compact, orientable, $n$-dimensional hyperbolic manifold, with a compact, orientable, totally geodesic hypersurface, $\iota: \Sigma \hookrightarrow M$, then,}
\[
H^{1}(\Gamma, S^{k}_{0} \mathbb{R}^{n+1}) \neq 0.
\]

\medskip

This theorem has been proven by Millson \cite{Millson1} using completely different methods
(the action of $\Gamma$ on trees) for the same set of possible
coefficients as us. For the case of coefficients in
$\mathbb{R}^{n+1}$, Lafontaine \cite{Lafontaine} gave a proof using
Riemannian geometry.  Our use of the Poincar\'{e} dual of totally geodesic submanifolds to give
elements of $H^{\bullet}( \Gamma, \mathbb{F})$ was inspired by the article
{\cite{Millson3}}.

\medskip

The authors would like to thank Andreas \v{C}ap for useful conversations.

\section{Projective differential geometry and hyperbolic manifolds.}

On a hyperbolic manifold it is particularly simple to construct the 
 Cartan bundle and the associated tractor
connections, and we show this here.

\subsection{The homogeneous model for projective geometry}\label{ProjectiveSphere}
We recall from the introduction:
\begin{defn}
A \underline{projective structure} $\pp$ on a manifold $M$ is an equivalence
class of torsion-free connections on $TM$, defined by the equivalence
relation: $\nabla \sim \tilde{\nabla}$ if and only if $\nabla$ and
$\tilde{\nabla}$ have the same unparameterised geodesics.
\end{defn}
Throughout, the group $SL(n+1,\mathbb{R})$ will be identified with its defining representation on $\mathbb{R}^{n+1}$ and usually denoted $G:=SL(n+1,\mathbb{R})$; 
the corresponding Lie
algebra will likewise be denoted $\mathfrak{g} := \mathfrak{sl}
(n+1,\mathbb{R})$.
We write $\mathbb{S}^{n}$ to denote the $n$-sphere equipped with 
the projective structure $[\nabla^r]$, where  $r$ is the round metric on the $n$-sphere and $\nabla^r$ the corresponding
Levi-Civita connection.
Now the $n$-sphere may be  identified with the set of oriented rays in
$\mathbb{R}^{n+1}$, so the group $G$ is seen to act on
$\mathbb{S}^{n}$, and sends great circles to great circles, so that
the projective structure is preserved. Conversely, any projective
transformation of an open connected subset of $\mathbb{S}^{n}$ can be
extended to an element of $G$.

 The stabiliser, in $G$, of the timelike ray
$\{(0,0,\ldots,0,t), t >0 \} \subset \mathbb{R}^{n+1}$ is a parabolic
subgroup $P \subset G$ and so the projective sphere $\mathbb{S}^{n}$
is a homogeneous space (in fact a flag manifold),
\[
\mathbb{S}^{n} \cong G / P.
\]
Thus the projective structure of $\mathbb{S}^{n}$ is closely tied 
to the properties of the group $G$, and the homogeneity of $G/P$. 
It is possible to formalise a way in which an
arbitrary manifold with a projective structure can be infinitesimally 
``compared'' to $\mathbb{S}^{n}$; the local failure of homogeneity is
measured in the curvature of an object called the \emph{Cartan
  connection}.

A manifold arising as a quotient of $\mathbb{S}^{n}$ by a discrete
group of projective transformations is clearly projectively equivalent
to $\mathbb{S}^{n}$ in a neighbourhood of any point, so in this case
the curvature of the Cartan connection vanishes. Discrete quotients of
this form are our objects of study, so here we will
 be dealing with flat Cartan connections. In the following
section we will discuss the aspects of 
Cartan geometry that we need.

\subsection{The Cartan bundle and connection}\label{Ct}
Observe that $G$ may be viewed as the total space of a principal $P$-bundle $G\to G/P\cong
\mathbb{S}^{n}$ that encodes the structure of $\mathbb{S}^{n}$ as a
homogeneous space. 
 Furthermore, recall that, as with any Lie group, the total space $G$ is canonically equipped with an equivariant
$\mathfrak{g}$-valued one form known as the Maurer-Cartan form
$\omega_{\rm MC}$; this is characterised by its identification of
left-invariant vector fields with elements of $T_eG$, the tangent
space at the identity.  Because of its left invariance, $\omega_{\rm MC}$ descends
to a well-defined $\mathfrak{g}$-valued one-form on discrete quotients
of $G$, with equivariance properties as stated here
(cf.\ \cite{Sharpe} chapter $4$, section $3$). 
\newcommand{\cG}{\mathcal{G}}
\newcommand{\cGG}{\mathcal{G}_{\Gamma}}
\newcommand{\om}{\omega}
\begin{lemma}\label{TractorConstruction}
If $\Gamma$ is a discrete subgroup of $G$, and $U
\subset G/P$ is an open set with the property that $M:=\Gamma \backslash U$
is a manifold, then $\cG_{\Gamma}:=\Gamma \backslash G$ is a principal $P$-bundle
over $\Gamma \backslash U$ and the Maurer-Cartan form $\omega_{\rm MC}$ induces a
$\mathfrak{g}$-valued one-form $\omega$ on $\cGG$ such
that,
\begin{enumerate}\label{axioms}
\item $R_{p}^{*} \omega = Ad(p^{-1}) \omega$ for all $p \in P$.
\item $\omega(u): T_{u}G \rightarrow \mathfrak{g}$ is a linear isomorphsim, for all $u \in \Gamma \backslash G$.
\item $\omega(\zeta_{X}(u))=X$, for all $u \in \Gamma \backslash G$, $X \in \mathfrak{p}$, and $\zeta_{X}$ the fundamental vector field generated by $X$.
\end{enumerate}
\end{lemma}
\begin{proof}
 We work over the inverse image of $U$ in $G$, and take $\cGG$ to mean
 the quotient of this by $\Gamma$.  Since $\Gamma$ acts on the left it
 is clear that $\cGG\to M$ is a principal bundle with fibre $P$.  Note
 that a vector field $X$ on $\cGG$ may be thought of as
 a vector field $\tilde{X}$ on $G$ with the property that
\begin{equation}\label{period}
T_\gamma (\tilde{X}_g)= \tilde{X}_{\gamma g} 
\end{equation}
for all $g\in G$ (in the open set over $U$) and $\gamma\in \Gamma$. 
Thus, if $[g]$ denotes the equivalence class of $g$ in $\cGG$, 
 we define $\omega([g])(X_{[g]})$ by 
$$
\omega ([g])(X_{[g]}):= \omega_{MC}( g)(\tilde{X}_{g}).
$$ 
This defines $\omega $ on $\cGG$ because, from (\ref{period}) and 
 the left invariance of $\omega_{MC}$, we have
$\omega_{MC}( g)(\tilde{X}_{g})=
\omega_{MC}(\gamma g)(\tilde{X}_{\gamma g})$ for any $\gamma \in \Gamma$.

Now the properties 1,2, and 3 follow at once from the equivariance of
$\omega_{\rm MC}$. 
\end{proof}

\begin{remk}\label{mgen} Note that the Lemma applies to any group $G$, with closed Lie subgroup $P$, and discrete subgroup $\Gamma$, where the latter  satisfies that $\Gamma\backslash U$ is a manifold $M$. 
The principal bundle structure and the properties 1,2, and 3 of the
Lemma mean that $(\cGG,\om)$ is a Cartan \textup{(}bundle, connection\textup{)} pair for
$M$.
\end{remk}

 For the case that $G=SL(n+1)$ and $P$ the parabolic subgroup,
 as described earlier, these groups identify the structure $(\cGG,\om)$ as a projective  Cartan
 connection. Furthermore the Cartan connection $\om$ given
here is the \emph{normal} Cartan connection for the bundle $\mathcal{G}_{\Gamma}$, in the sense of \cite{CS}, 
because $\omega$ agrees  locally with
 the Maurer-Cartan form. Note that a Cartan connection differs
from the usual notion of a principal $P$-connection; for example the
$\mathfrak{g}$-valued one-form $\omega$ has no kernel, so there is no
horizontal distribution.

Given finite-dimensional $P$-representation $\rho:
P \rightarrow GL(\mathbb{E})$ 
we can  form an
associated vector bundle over $M$,
\begin{equation}\label{ind}
\mathcal{G}_{\Gamma} \times_{\rho} \mathbb{E}.
\end{equation}
For example, it follows easily from Lemma \ref{TractorConstruction}
that the tangent bundle arises from the case that $\mathbb{E}$ is
$\frak{g}/\frak{p}$, where $\frak{p}:=\operatorname{Lie}(P)$ and
$\rho$ is induced from the restriction to $P$ of the Adjoint
representation. Thus all tensor bundles arise as associated bundles of
the form (\ref{ind}).  Unfortunately the properties of the Cartan
connection $\om$ mean that it does not in general induce a linear
connection on these induced vector bundles.  However there are such
connections on a special class, that we now describe.

\subsection{Tractor bundles and the tractor connection}\label{tract}
A special case of the associated vector bundle construction arises if
$\mathbb{E}$ carries a $G$-representation $\rho: G \rightarrow
GL(\mathbb{E})$, rather than just a $P$-representation. In this case the associated vector bundles 
\[
\mathcal{T}_{\mathbb{E}}:= \mathcal{G}_{\Gamma} \times_{\rho} \mathbb{E},
\]
 are called \emph{projective tractor bundles}
 \cite{BailEastGov,CapGover}. Tractor bundles exist in far greater generality,
 but here we shall only be concerned with those arising in connection
 with propjective geometry, so the adjective ``projective'' will usually be
 omitted in the following. 

In contrast to general associated vector bundles as in (\ref{ind}),
the Cartan connection $\omega$ does induce a linear connection on
tractor bundles \cite[Theorem 2.7]{CapGover}, so that tractor bundles
are of considerable importance. In particular for the connection from
Lemma \ref{TractorConstruction} we have:
\begin{lemma}\label{CapGoverConnection}
For any finite-dimensional $G$-module $\rho: G \hookrightarrow GL(\mathbb{E})$, the Cartan connection $\om$ on $\mathcal{G}_{\Gamma}$ induces a flat linear connection $\nabla^{\mathcal{T}_{\mathbb{E}}}$ on the tractor bundle $\mathcal{T}_{\mathbb{E}}$.
\end{lemma}
\begin{proof}(Sketch)
Sections of $\mathcal{T}_{\mathbb{E}}$ can be identified with
$P$-equivariant maps from $\mathcal{G}_{\Gamma}$ to $\mathbb{E}$.  For
any such map $\sigma$, and any $X \in \Gamma(TM)$, choose an arbitrary
lift $\tilde{X}$ of $X$ to $\mathcal{G}_{\Gamma}$ and set,
\[
\widetilde{\nabla^{\mathcal{T}_{\mathbb{E}}}_{X} \sigma } := \tilde{X} \cdot \sigma + \rho(\omega(\tilde{X})) \sigma,
\]
where the action $\cdot$ is just the derivation action of a vector
field on an $\mathbb{E}$-valued function. It can be shown
\cite{CapGover}, that the choice of lift $\tilde{X}$ does not matter,
and that this definition gives a linear connection.

The fact that this connection is flat follows easily from the fact
that, locally, the Maurer-Cartan form $\omega_{\rm MC}$ satisfies the
Maurer-Cartan equations.
\end{proof}
The connection $\nabla^{\mathcal{T}_{\mathbb{E}}}$ on $\mathcal{T}_{\mathbb{E}}$ is called the \emph{(projective) tractor connection}. Apart from the data of the particular representation $\rho$, by construction this is
determined solely by the structure of the homogeneous space $G/P$
and the action of the group  $\Gamma$ on this. So it 
 depends only on the \emph{projective
  geometry} of $M$.

\section{Hyperbolic manifolds and Projective Geometry}\label{hyp-Sec}
We now fit hyperbolic manifolds into the setup developed above. We do
this by first showing that, up to isometry, any hyperbolic manifold 
arises from 
a discrete quotient of an open $SO_{o}(n,1)$ orbit in the projective sphere
$\mathbb{S}^{n}$, where $SO_{o}(n,1)$ is the identity connected
component of $SO(n,1)$.

Once this has been achieved, we  relate the corresponding
projective tractor bundle to the canonical flat bundle familiar in
hyperbolic geometry.
 
\subsection{The Beltrami-Klein model and the projective $n$-sphere}\label{FlatModels}

We use a slight variant of the usual Beltrami-Klein model, as follows.
Let $h$ be the diagonal Lorentzian metric of signature $(n,1)$ on
$\mathbb{R}^{n+1}$ such that, for $x = (x_{1} , \ldots, x_{n+1}) \in \mathbb{R}^{n+1}$,
\[
h(x,x) = x_{1}^{2} + \cdots + x_{n}^{2}  - x_{n+1}^{2}.
\]
 Consider the hyperboloid of two sheets $ \{ x \in
\mathbb{R}^{n+1} \; s.t.\;  h(x,x)=-1 \}$. We will write $\mathbf{I}$ to
denote the positive sheet of this; that is, the sheet consisting of
points with positive last coordinate. This projects onto an open
subset of the space of rays in $\mathbb{R}^{n+1}$. Recall that the space
of rays in $\mathbb{R}^{n+1}$ is naturally diffeomorphic to a standard  $n$-sphere
$\mathbb{S}^{n}$, and the projection, $\mathbb{P}_+(\mathbf{I})$, of
$\mathbf{I}$ is a `cap' of this sphere. The pullback of $h$ along the
projection is a signature $(n,0)$ metric on the cap: this gives the
cap the structure of $n$-dimensional hyperbolic space
$\mathbb{H}^{n}$. Let $\isom(\mathbb{H}^{n})$ denote the group of
orientation-preserving isometries of the Riemannian manifold
$\mathbb{H}^{n}$.

We recover the group picture associated with these structures as
follows. As mentioned earlier we identify $ G=SL(n+1,\mathbb{R})$ with
its linear representation on $\mathbb{R}^{n+1}$.  We identify
$SO(n,1)$ with the subgroup of this linear group preserving the metric
$h$, as an indefinite inner product on $\mathbb{R}^{n+1}$; so
$SO_{o}(n,1)$ is the identity connected component therein.  Then
$\isom(\mathbb{H}^{n}) \subseteq SO_{o}(n,1)$, 
so the
orientation-preserving isometries of the cap form a class of
projective transformations of $\mathbb{S}^{n}$.

\begin{defn} We fix here some notation. Henceforth, $M$ denotes a compact, orientable,
  hyperbolic manifold. We
  use $g$ to denote the hyperbolic metric on $M$, and $\nabla^{g}$ is
  the corresponding Levi-Civita connection on $TM$. We sometimes also
  use $g$ to denote group elements, but no confusion should arise as
  the meaning will always be clear by the context.
\end{defn}

The \emph{developing map} (see, for example, Chapter $B$ of \cite{BenedettiPetronio}) 
gives a faithful representation,
\[
dev: \pi_{1}(M) \rightarrow \isom(\mathbb{H}^{n}),
\]
so that $dev(\pi_{1}(M))$ is a discrete, torsion-free, cocompact subgroup $\Gamma \subset \isom(\mathbb{H}^{n})$, 
and $M$ is isometric to $\Gamma \backslash \mathbb{H}^{n}$; we identify  $M$ with $\Gamma \backslash \mathbb{H}^{n}$.

If $\Gamma $ is such a group, then
evidently $\Gamma$ also embeds in the group of projective
transformations of $\mathbb{S}^{n}$. We identify the group $\Gamma$
with its image under the inclusions $\Gamma \subset SO_{o}(n,1)\subset
G$. Now recall that $G$ acts transitively on $\mathbb{S}^{n}$ and we
write $P$ for the parabolic subgroup fixing a nominated point, so we
may identify $\mathbb{S}^{n}=\mathbb{P}_+(\mathbb{R}^{n+1})$ with
$G/P$, see Section \ref{ProjectiveSphere}.  We write $U:=
\mathbb{P}_+(\mathbf{I})$ and observe this is an open $SO_{o}(n,1)$
orbit of $\mathbb{S}^{n}$.  Then $(U,g)$ is identified with $\mathbb{H}^n$, where $g$ is
the pullback of the metric on $\mathbf{I}$ (and see the remark
below).  As $\Gamma $ acts by
isometries on $\mathbb{H}^n$ the metric descends to the hyperbolic
metric $g$ on $M \cong \Gamma \backslash \mathbb{H}^{n}$.

Since $\Gamma$ is a subgroup of $G$, Lemma
\ref{TractorConstruction} gives that the $P$-bundle
$\mathcal{G}_{\Gamma}$ restricts to a smooth $P$-bundle over $\Gamma
\backslash \mathbb{H}^{n}$. By construction this is the Cartan bundle
for the projective structure determined by the Levi-Civita connection
of $g$, and Lemma \ref{CapGoverConnection} gives the associated
projective tractor bundles and tractor connections on $\Gamma
\backslash \mathbb{H}^{n}$, for any $G$-representation $\mathbb{E}$.

\begin{remk}
The construction above identifies the hyperbolic manifold $(M,g)$ as a
projective manifold $(M,\pp)$ equipped with a Cartan holonomy
reduction, in the sense of \cite{CapGoverHammerl,CGH}. The inner product $h$ induces a parallel signature $(n,1)$ metric on the  tractor bundle $\mathcal{T}_{\mathbb{R}^{n+1}}$ and this gives the holonomy reduction
in this case. As explained in those sources \textup{(}and
see also \cite{Armstrong,GoverMacbeth}\textup{)} such a reduction is closely
linked to an Einstein metric on the underlying manifold. In the case
here this is the hyperbolic metric $g$.
\end{remk}

\subsection{The standard tractor bundle and its dual}

Setting $\mathbb{E}$ to be the defining $G$-representation
$\mathbb{R}^{n+1}$, gives the {\em standard tractor bundle} over
$M$. This will be denoted simply $\mathcal{T}$, rather than
$\mathcal{T}_{\mathbb{R}^{n+1}}$.  Since $\mathcal{T}$ is an
associated bundle to $\mathbb{R}^{n+1}$, tensor products of
$\mathcal{T}$ and its dual are associated vector bundles for tensor
products of $\mathbb{R}^{n+1}$ and its dual.  The notation
$\nabla^{\mathcal{T}}$ will be used for the tractor connection on any
of these.

The basic projective tractor calculus is developed in  
\cite[Section 3]{BailEastGov}. See also \cite{CapGoverMacbeth,Eastwood,GoverMacbeth}.
There is a canonical isomorphism of vector bundles,
\[
\mathcal{T} \cong TM \oplus \mathbb{R},
\]
where $\mathbb{R}$ denotes the trivial line bundle.
For readers who are familiar with projective tractor calculus, in this
isomorphism we have used the canonical scale arising from the
hyperbolic metric to split the tractor bundle and trivialise
projective density bundles.

Using this isomorphism, and the fact that the connection $\omega$ on $\mathcal{G}_{\Gamma}$ is the normal projective Cartan connection, we can use the 
formula from section $3$ of \cite{BailEastGov} to get that the action of the tractor connection is,
\begin{equation}\label{TracFormula}
\nabla^{\mathcal{T}}_{X}
\left( 
\begin{array}{c}
Y \\
 s \\
\end{array}
\right)
=
\left(
\begin{array}{c}
\nabla^{g}_{X}Y + sX \\
X(s) + g(X,Y) \\
\end{array}
\right),  
\end{equation}
where $X,Y \in \Gamma(TM)$ and $s \in C^{\infty}(M)$. In this formula we have also used that the projective Schouten tensor
$\mathbf{P}$ (for the metric projective structure) on a hyperbolic
manifold satisfies $\mathbf{P}=-g$.

The dual of the defining representation $(\mathbb{R}^{n+1})^{*}$
induces the dual standard tractor bundle on $M$. For this there is a similar
canonical isomorphism,
\[
\mathcal{T}^{*} \cong \mathbb{R} \oplus TM^{*},
\]
and a corresponding expression for the tractor connection,
\[
\nabla_{X}^{\mathcal{T}^{*}}
\left(
\begin{array}{c}
s \\ 
\eta \\
\end{array}
\right)
=
\left(
\begin{array}{c}
X(s) -\eta(X)\\
\nabla_{X} \eta - sX^{\flat} \\
\end{array}
\right),
\]
for $\eta \in \Gamma(TM^{*})$, $s \in C^{\infty}(M)$, $X \in \Gamma(TM)$, and where
$X^{\flat} = g(\cdot,X)$.

The vector bundle decompositions and formulae in these examples imply
similar formulae for the tractor bundles over $M$
associated to any tensor product of $\mathbb{R}^{n+1}$ and its dual.

\subsection{The canonical flat connection and the tractor connection}\label{locvsTr}
Now we will construct the usual canonical flat connection for
hyperbolic geometry, as mentioned in the introduction. In fact,
although we will not use it here, the construction applies far more
generally and is standard in the study of locally symmetric spaces,
(see, for example, \cite{MatsushimaMurakami}).

Given a faithful $SO(n,1)$-module $\mathbb{F}$, composition with $dev:
\pi_{1}(M) \hookrightarrow \isom(\mathbb{H}^{n})$ yields a
representation (which we will still denote $dev$), $dev: \pi_{1}(M)
\hookrightarrow GL(\mathbb{F})$. This gives rise to a vector bundle $\pi: F \rightarrow M$ with total space
\[
F := \mathbb{H}^{n} \times_{dev} \mathbb{F},
\]
where the universal covering space $\mathbb{H}^{n}$ is to be thought of
as a left principal $\pi_{1}(M)$-bundle over $M$. The notation on the
right-hand-side is understood as follows.  The group $\pi_{1}(M)\cong
\Gamma$ acts on $\mathbb{H}^{n}$ on the left, and we are using $\gamma \cdot x$
to indicate the action of an element of $\gamma \in \Gamma \cong
\pi_{1}(M)$ on a point $x \in \mathbb{H}^{n}$. Then  
\[
\mathbb{H}^{n} \times_{dev} \mathbb{F}: = (\mathbb{H}^{n} \times \mathbb{F}) /  \sim, 
\]
with the equivalence relation $\sim$ on points $(x,v) \in
\mathbb{H}^{n} \times \mathbb{F}$ defined by,
\[
(x,v) \sim (\gamma \cdot x, dev( \gamma ) v), \; \forall \gamma \in \Gamma.
\]
This differs from the usual quotient in the associated vector bundle
construction because of the left, rather than right action of
$\pi_{1}(M)$. By construction $F$ is thus equal to $\Gamma \backslash
(\mathbb{H}^{n} \times \mathbb{F})$, with the left action of $\Gamma$ as in
the definition of the equivalence relation $\sim$.

Note that a section $f\in \Gamma(\mathbb{H}^{n} \times_{dev}
\mathbb{F})$ is equivalent to a function $\tilde{f}:\mathbb{H}^{n} \to
\mathbb{F}$ with the equivariance property,
\begin{equation}\label{equiv}
\tilde{f}(\gamma\cdot x)= dev(\gamma) \tilde{f}(x),
\end{equation}
for all $x\in\mathbb{H}^{n} $ and for all $\gamma\in \Gamma$.
Similarly a tangent vector field $v$ on $M$ is represented by
$\tilde{v}\in \Gamma(T\mathbb{H}^n)$ such that
$$
\tilde{v}_{\gamma \cdot x}= T_\gamma (\tilde{v}_x) .
$$ 
It follows that the derivative $\tilde{v} \tilde{f}$, of $\tilde{f}$, is again equivariant in the sense of (\ref{equiv}), that is 
$$
(\tilde{v}\tilde{f})(\gamma\cdot x)= dev(\gamma) (\tilde{v}\tilde{f})(x),
$$ for all $x\in\mathbb{H}^{n} $ and for all $\gamma\in \Gamma$.  This
means that the canonical trivial connection on $\mathbb{H}^{n} \times
\mathbb{F}$ descends to a (flat) connection $\nabla^{dev}$ on $F$,
over $M$. We shall call the pair $(F,\nabla^{dev})$ the canonical flat
connection associated to $\mathbb{F}$. By construction, the holonomy
of the connection $\nabla^{dev}$ is, up to conjugation, the faithful
representation of $\Gamma$ on $\mathbb{F}$.

Recall that $M$ can be thought of as a subset of $\Gamma \backslash
\mathbb{S}^{n}$, and $\Gamma\backslash G$ restricts to give the smooth
$P$-bundle $\mathcal{G}_{\Gamma}$ over $M$.  We used this in Lemma
\ref{CapGoverConnection} to define tractor bundles.
\begin{thm}\label{MainIsom}
If $\rho: SL(n+1,\mathbb{R}) \rightarrow GL(\mathbb{E})$ is a finite dimensional representation, and $\mu: SO(n,1) \rightarrow GL(\mathbb{F})$ is a finite dimensional 
representation with $\rho|_{SO(n,1)} \cong \mu$, then there is a canonical isomorphism of flat vector bundles,
\[
(F,\nabla^{dev}) \stackrel{\simeq}{\longrightarrow} (\mathcal{T}_{\mathbb{E}},\nabla^{\mathcal{T}_{\mathbb{E}}}).
\]
\end{thm}
\begin{remk}
Even if the representation $\rho$ used to define the tractor bundle is irreducible, the representation $\mu$ defining the \textup{(}hyperbolic\textup{)} canonical flat bundle will, typically, be reducible.
\end{remk}

\begin{proof} 
We need to first show that, over $M$, there is an isomorphism of vector bundles 
\begin{equation}\label{vis}
\Gamma \backslash (\mathbb{H}^{n} \times \mathbb{E}) \stackrel{\simeq}{\longrightarrow} \mathcal{G}_{\Gamma} \times_{\rho} \mathbb{E},
\end{equation} 
where 
$\gamma \in \Gamma$ acts on  $\mathbb{H}^{n} \times \mathbb{E}$ by,
\[
\gamma (x,y) = (\gamma \cdot x, \gamma v),
\]
and we have used  the embedding $\Gamma \hookrightarrow SL(n+1,\mathbb{R})$ 
to define the action of $\gamma$ on $v \in \mathbb{E}$.
Furthermore, we will show that this is compatible with the connections, as stated.

To do this, we use that there is a vector bundle isomorphism
$\Phi: (G/P) \times \mathbb{E} \rightarrow G \times_{\rho} \mathbb{E}$
given by,
\[ 
\Phi:  (gP, v) \mapsto [(g,g^{-1}v)]_{P}.
\]
It is easily checked that this map is well defined. The inverse map is,
\[
\Phi^{-1}: [(g,v)]_{P} \mapsto (gP, gv).
\]
Next we use the fact that $\mathbb{H}^n$ is identified with an open subset of
$G/P$, so that points $(x,v) \in \mathbb{H}^n \times \mathbb{E}$ can
be written $(gP,v)$. The required isomorphism is constructed by first
forming the mapping $\mathbb{H}^n \times \mathbb{E} \to G
\times_{\rho} \mathbb{E}$ given by
\[
(x,v) = (gP,v) \mapsto \Phi (gP,v),
\]
and then mapping the image to the quotient $\mathcal{G}_\Gamma
\times_\rho \mathbb{E}$ by the projection from $G \times_{\rho}
\mathbb{E}$ to $\mathcal{G}_\Gamma \times_\rho \mathbb{E} = (\Gamma
\backslash G) \times_{\rho} \mathbb{E}$.  It remains 
to show that 
this mapping $\mathbb{H}^n\times \mathbb{E}\to \mathcal{G}_\Gamma \times_\rho \mathbb{E}$ descends to a well-defined
map on $\Gamma \backslash (\mathbb{H}^n \times \mathbb{E})\to \mathcal{G}_\Gamma \times_\rho \mathbb{E}$.  But this
follows at once  from the fact that,
\[
\Phi (\gamma g P, \gamma v) = [ (\gamma g, g^{-1} v)]_P.
\]
So we have the bundle map (\ref{vis}), and by the construction this is an isomorphism. 

Via the isomorphism $\Phi$,  the trivial connection 
on the bundle $(G/P)\times \mathbb{E}$ over $\mathbb{H}^n$ is mapped to 
that tractor connection $\nabla^{\mathcal{T}_{\mathbb{E}}}$ on $G\times_\rho \mathbb{E}$. 
 This implies
that the vector bundle isomorphism $\Gamma \backslash (\mathbb{H}^{n}
\times \mathbb{E}) \stackrel{\simeq}{\to} \mathcal{T}_{\mathbb{E}}$ sends $\nabla^{dev}$ to
$\nabla^{\mathcal{T}_{\mathbb{E}}}$.
\end{proof}

\subsection{The Tractor connection and group cohomology}\label{VoganZuckerman}

Given an $SO(n,1)$-module $\mathbb{F}$, we can use the developing map $dev$ to construct the group cohomology, $H^{\bullet}(\Gamma, \mathbb{F})$ (see \cite{Brown}).
Since $M$ is an Eilenberg-Maclane space $K (\Gamma,1)$ the group cohomology is isomorphic to the simplicial cohomology of $M$ with local coefficients in $\mathbb{F}$. 

The connection $\nabla^{dev}$ is flat, so a differential complex on $M$ arises by twisting the de Rham complex with $(\nabla^{dev}, F)$. It is a standard fact (see, for example, Section $3$ of \cite{Millson3}) that the cohomology of this twisted de Rham complex is isomorphic to the simplicial cohomology of $M$ with local coefficients in $\mathbb{F}$, and hence is isomorphic to $H^{\bullet}(\Gamma, \mathbb{F})$.

Next, using that the isomorphism in Theorem \ref{MainIsom} is also an
isomorphism of connections (and with $E$ and $F$ related as there), we
have at once the following important consequence. Here, as in Theorem
\ref{MainIsom}, $\mathbb{F}$ is a restriction to $\Gamma\subset
SO(n,1)$ of the representation $\mathbb{E}$ of $SL(n+1)$, inducing
$\cT_E$.
\begin{prop} \label{preview}
 The cohomology of the twisted de Rham complex,
\begin{equation}\label{twisteddeRham0}
0 \longrightarrow \Gamma( \mathcal{T}_{\mathbb{E}} ) \overset{ \nabla^{\mathcal{T}_{\mathbb{E}}} }{\longrightarrow} \Gamma (TM^{*} \otimes \mathcal{T}_{\mathbb{E}}) \overset{ d^{\nabla^{\mathcal{T}_{\mathbb{E}}}} }{\longrightarrow} \ldots \overset{ d^{\nabla^{\mathcal{T}_{\mathbb{E}}}}}{\longrightarrow} \Gamma (\Lambda^{n} TM^{*} \otimes \mathcal{T}_{\mathbb{E}} ) \longrightarrow 0,
\end{equation}
is isomorphic to the group cohomology with coefficients $H^{\bullet}(\Gamma, \mathbb{F})$.
\end{prop}

\begin{exam}
If $\mathbb{E}=S^{2} \mathbb{R}^{n+1}$, then, 
\begin{equation}\label{twisteddeRham1}
H^{\bullet}(M, \mathcal{T}_{S^{2} \mathbb{R}^{n+1}}) \cong H^{\bullet} (\Gamma, S^{2}_{0} \mathbb{R}^{n+1}) \oplus H^{\bullet}(\Gamma, \mathbb{R}^{n+1}),
\end{equation}
where $S^{2}_{0} \mathbb{R}^{n+1}$ denotes the trace-free part (using the $(n,1)$ metric $h$) of $S^{2} \mathbb{R}^{n+1}$, and $H^{\bullet}(M, \mathcal{T}_{S^{2} \mathbb{R}^{n+1}})$ denotes the cohomology of the twisted de Rham complex (\ref{twisteddeRham0}). 
\end{exam}

The relationship to group cohomology means that we can use the vanishing results of \cite{VoganZuckerman}. To do so we introduce some notation. 
Fix a Cartan subalgebra $\mathfrak{h}$ of the complexification of $\mathfrak{so}(n,1)$. Let $\varepsilon_{i}$ be the $i$th usual Cartesian coordinate functional on $\mathfrak{h}$, so that the positive roots are the sums and differences of the $\varepsilon_{i}$ in case $n+1$ is even, and $\varepsilon_{i}$ are the short positive roots if $n+1$ is odd.
If $m$ is the rank of $SO(n,1)$, and $\lambda=(\lambda_{1}, \lambda_{2}, \ldots, \lambda_{m})$ is a weight expressed in the basis $\{ \varepsilon_{i} \}_{i=1}^{m}$, define $i(\lambda)$ to be the number of nonzero entries in the vector $\lambda$.
\begin{thm}[\cite{VoganZuckerman, Millson3}]\label{Vanishing}
If $\mathbb{F}$ is an irreducible representation of $SO(n,1)$ with highest weight $\lambda$, then,
\begin{enumerate}
\item If $n=2m-1$ and all entries of $\lambda$ are nonzero, then $H^{k}(\Gamma, \mathbb{F})=0$ for all $k$.
\item For all other irreps $\mathbb{F}$,
\[
k \notin \{ i (\lambda), i (\lambda) +1, \ldots, n - i (\lambda) \} \implies H^{k} (\Gamma, \mathbb{F}) =0.
\]
\end{enumerate}
\end{thm}
\begin{remk}
This statement of what the Vogan-Zuckerman theorem implies for hyperbolic manifolds is taken from \cite{Millson3}.
\end{remk}
An important instance of this vanishing theorem that we will use later is,
\begin{equation}\label{AllVanish}
H^0(\Gamma, \mathbb{F})=0,
\end{equation}
for any nontrivial representation $\mathbb{F}$.

As a consistency check, and to provide an alternative geometric perspective, we include a short tractor calculus proof of the case where $\mathbb{F}=\mathbb{R}^{n+1}$.
\begin{lemma}\label{NoLine}
On a compact hyperbolic manifold $M = \Gamma \backslash \mathbb{H}^{n}$, 
\[
H^{0}(\Gamma, \mathbb{R}^{n+1})=0.
\]
\end{lemma}
\begin{proof}
From theorem \ref{MainIsom} it is sufficient to show that there are no (nontrivial) globally parallel sections of $\mathcal{T}$. 
Suppose, with a view to contradiction, that 
\[
\sigma = \left(
\begin{array}{c}
Y \\
s
\end{array}
\right) \; 
\in \Gamma(\mathcal{T}),
\]
is a nontrivial section satisfying $\nabla^{\mathcal{T}} \sigma=0$ everywhere on $M$. Then, for each $X\in \Gamma(TM)$, formula (\ref{TracFormula}) gives two equations,
\begin{align}\nonumber
\nabla^{g}_{X} Y + s X &= 0 \\ \nonumber
\nabla_{X}^{g} s + g(X,Y) &=0. \\ \nonumber 
\end{align}
The bottom equation means that $ds=-Y^{\flat}$, and so $s$ must be nontrivial, since $\sigma$ is assumed to be nontrivial. 
Substituting $ds=-Y^{\flat}$ into the first equations yields, in indices, 
\[
-\nabla_{a} \nabla^{b} s + s \tensor{\delta}{_a^b}=0,
\] 
and contraction gives that $\Delta^{g} s + n s =0$. This is a  contradiction because $\Delta^{g}$ has nonnegative eigenvalues.
\end{proof}

\begin{remk}\label{thomasD}
For readers who are familiar with tractor calculus, a similar proof
can be given that $H^{0}(\Gamma, S^{k}_{0} \mathbb{R}^{n+1}) = 0$,
using the Thomas $D$ operator and the canonical tractor $X$
\textup{(}\cite{BailEastGov}, Section $3.5$\textup{)}. Suppose a
non-zero parallel, trace-free symmetric tractor $\sigma_{A_{1} A_{2} \cdots
  A_{k}} \in \Gamma(S^{k}_{0} \mathbb{T}) $ exists. Then $s := \left(
\sigma_{A_{1} A_{2} \cdots A_{k}} X^{A_{1}} X^{A_{2}} \cdots X^{A_{k}}
\right)$ is not zero, where we have used abstract indices. Using the
pseudo-Leibniz rule which the Thomas $D$ operator satisfies
\textup{(}eg. in the proof of Proposition $3.6$ of
\cite{BailEastGov}\textup{)}, it is easy to show,
\[
D^{A}D_{A} s = 0.
\]
On the other hand,
\[
D^{A} D_{A} s = -\Delta^{g} s  -  k (k+n-1) s,
\]
so that $\Delta^{g}$ has a negative eigenvalue $-k(k+n-1)$, which is
impossible. In fact, a variant of this tractor calculus proof shows
that the result holds much more generally: for example no parallel sections of
$S^{k}_{0} \mathcal{T}$ exist on any closed
manifold of negative constant scalar curvature.
\end{remk}

\section{Bernstein-Gelfand-Gelfand Complexes}
\label{BGGsec}

An important motivation for constructing the isomorphism between the
flat bundle $(E, \nabla^{dev})$ and $(\mathcal{T}_{\mathbb{E}},
\nabla^{\mathcal{T}_{\mathbb{E}}})$ is that the latter is intimately
related to the construction of projectively invariant differential
operators on $M$. Moreover, we can use the tractor connection to
construct \emph{complexes} of differential operators, generalising the
Bernstein-Gelfand-Gelfand (BGG) complexes of representation theory.

We will now give a brief outline of the construction of these
differential BGG complexes, using the tractor connection.  Unless
stated otherwise, proofs of all statements in the following two
sections can be found in the foundational articles \cite{CSS}
\cite{CaldDiemer}, (also \cite{CS2}, section $2$ contains a useful
review of these constructions).

\subsection{The $|1|$-grading and Lie algebra homology}\label{BG}
It will be necessary to utilise some features of the group $G$ and its Lie algebra. 

First of all, the Lie algebra $\mathfrak{g}$ has a decomposition,
\[
\mathfrak{g} \cong \mathfrak{g}_{-1} \oplus \mathfrak{g}_{0} \oplus \mathfrak{g}_{1},
\]
where $\mathfrak{p}:=\mathfrak{g}_{0} \oplus \mathfrak{g}_{1}$ is the Lie algebra of the parabolic subgroup $P$, and $\mathfrak{g}_{1}$ is a nilpotent ideal of $\mathfrak{p}$.
The Lie algebra direct summands have the property that $[\mathfrak{g}_{i}, \mathfrak{g}_{j}] \subset \mathfrak{g}_{i+j}$, so that $\mathfrak{g}$ is a `$|1|$-graded Lie algebra'. 
For more information on $|k|$-graded Lie algebras, consult section $3.1.2$ of \cite{CS}.

Fix a finite-dimensional $SL(n+1,\mathbb{R})$ representation $\rho: SL(n+1,\mathbb{R}) \rightarrow GL(\mathbb{E})$. Define the vector spaces $C^{k}(\mathfrak{g}_{1},
\mathbb{E}) := \Lambda^{k} \mathfrak{g}_{1} \otimes \mathbb{E}$; these have a natural $P$-module structure, and we will use the same notation for the corresponding representation 
$\rho: P \rightarrow GL(\Lambda^{k} \frak{g}_{1} \otimes \mathbb{E})$. 

{}From this $P$-module structure it follows that we can form
associated vector bundles, $\mathcal{G}_{\Gamma} \times_{\rho} C^{k}
(\mathfrak{g}_{1}, \mathbb{E})$.  As mentioned in section \ref{Ct}, it
is easily verified, via the adjoint representation, that the
associated vector bundle $\mathcal{G}_{\Gamma} \times_{Ad}
\mathfrak{g} / \mathfrak{p}$ is isomorphic to the tangent bundle of
$M$, see e.g.\ \cite{CSS}. Since $(\mathfrak{g}/ \mathfrak{p})^{*}
\cong \mathfrak{g}_{1}$ as $P$-modules, it can then also be shown
that,
\[
\mathcal{G}_{\Gamma} \times_{\rho} C^{k} (\mathfrak{g}_{1}, \mathbb{E}) \cong \Lambda^{k} TM^{*} \otimes \mathcal{T}_{\mathbb{E}}.
\]

Now we define the `Kostant codifferential' $\partial^{*}: C^{k}(\mathfrak{g}_{1}, \mathbb{E}) \rightarrow C^{k-1}(\mathfrak{g}_{1}, \mathbb{E})$ by,
$$
\partial^{*} (Z_{1} \wedge \ldots Z_{k} \otimes v) = \sum_{i=0}^{k} (-1)^{i+1} Z_{1} \wedge \widehat{Z}_{i} \ldots \wedge Z_{n} \otimes Z_{i} \cdot v ,
$$ where $Z_{i} \in \mathfrak{g}_{1}$ and $v \in \mathbb{E}$. We get
that $(\partial^{*})^{2}=0$, and also that the homology vector spaces
given by this differential are naturally $P$-modules, which will be
denoted $H_{k}(\mathfrak{g}_{1}, \mathbb{E})$. Again, the $P$-module
action will still be denoted $\rho$, since it is determined by the
original representation $\rho$.

The Kostant codifferential is $P$-equivariant, and so induces a
canonical bundle map (which will be denoted with the same symbol
$\partial^{*}$) between the associated bundles,
\[
\partial^{*}: \Lambda^{k} TM^{*} \otimes \mathcal{T}_{\mathbb{E}} \rightarrow \Lambda^{k-1} TM^{*} \otimes \mathcal{T}_{\mathbb{E}}.
\]
The quotient bundles that arise as the homology of $\partial^{*}$ at twisted $k$-forms 
will be denoted $H_{k}(E)$. By construction, we have that $H_{k}(E) \cong \mathcal{G}_{\Gamma} \times_{\rho} H_{k} (\mathfrak{g}_{1}, \mathbb{E})$. 

There is a natural bundle map, 
\begin{equation}\label{proj}
\pi: Ker( \partial^{*} ) \rightarrow H_{k}(E),
\end{equation}
and the same symbol will be used for the corresponding  map on sections. 
\begin{exam}\label{dualex}
If $\mathbb{E}=(\mathbb{R}^{n+1})^{*}$ is the dual of the $G$-defining representation  then $\mathcal{T}_{\mathbb{E}}=\mathcal{T}^{*}$ is the dual of the standard 
tractor bundle, and a simple section in $\Lambda^{k} TM^{*} \otimes \mathcal{T}^{*}$ can be written,
\[
\left(
\begin{array}{c} 
\omega s \\
\omega \otimes \eta \\
\end{array}
\right),
\]
for $\omega \in \Gamma( \Lambda^{k}  TM^{*} ))$, $\eta \in \Gamma (TM^{*})$ and $s \in C^{\infty}(M)$. The Kostant codifferential $\partial^{*}$, is given by,
\[
\partial^{*} 
\left(
\begin{array}{c}
 \omega s \\
\omega \otimes \eta \\
\end{array}
\right) = 
\left(
\begin{array}{c} 
0 \\
\omega s \\
\end{array}
\right).
\]
Clearly, $(\partial^{*})^{2}=0$, and the first three homology bundles are,
\[
H_{0}((\mathbb{R}^{n+1})^{*})=\mathbb{R}, \; H_{1}((\mathbb{R}^{n+1})^{*})=S^{2} TM^{*}, \; and \; H_{2}((\mathbb{R}^{n+1})^{*}) = \Lambda^{2} TM^{*} \odot TM^{*},
\]
where (the `Cartan part')  $\Lambda^{2} TM^{*} \odot TM^{*}$ is 
$(\Lambda^{2}TM^{*} \otimes TM^{*}) / \Lambda^{3}TM^{*}$. 
\end{exam}

\subsection{The BGG operators}
\label{BGGSec}

For each irreducible $SL(n+1)$ representation $\mathbb{E}$, the tractor
connection $\nabla^{\mathcal{T}_{\mathbb{E}}}: \Gamma ( \mathcal{T}_{\mathbb{E}}) \rightarrow \Gamma
(TM^{*} \otimes \mathcal{T}_{\mathbb{E}})$ is flat, so there is a corresponding
twisted de Rham complex,
\begin{equation}\label{TwisteddeRham}
0 \longrightarrow \Gamma(\mathcal{T}_{\mathbb{E}})
\overset{\nabla^{\mathcal{T}_{\mathbb{E}}}}{\longrightarrow} \Gamma(TM^{*} \otimes
\mathcal{T}_{\mathbb{E}}) \overset{d^{\nabla^{\mathcal{T}_{\mathbb{E}}}}}{\longrightarrow} \ldots
\overset{ d^{ \nabla^{\mathcal{T}_{\mathbb{E}}}}}{ \longrightarrow } \Gamma(\Lambda^{n}
TM^{*} \otimes \mathcal{T}_{\mathbb{E}}) \longrightarrow 0,
\end{equation}
as used in Proposition \ref{preview}. 
The groundbreaking works \cite{CSS}, \cite{CaldDiemer}, (building on
\cite{Baston1}, \cite{Baston2}), give a construction of a
corresponding differential complex involving lower rank bundles as
follows.

There is, for each $k\in \{0,1,\cdots ,n\}$, a differential  {\em BGG splitting
operator} $L_k: H_{k}(E) \rightarrow \Lambda^{k}TM^{*} \otimes
\mathcal{T}_{\mathbb{E}}$, uniquely characterised by the following properties, for
$\sigma \in H_{k}(E)$,
\begin{enumerate}
\item $\partial^{*} L_k \sigma =0$
\item $\partial^{*} d^{\nabla^{\mathcal{T}_{\mathbb{E}}}} L_k \sigma =0$
\item $\pi L_k \sigma = \sigma$, where $\pi$ is the projection (\ref{proj}) from $\Lambda^{k} TM^{*} \otimes \mathcal{T}_{\mathbb{E}}$ to $H_{k} (E)$.
\end{enumerate}
Next, using the operators $L_k$, it is evidently possible to define differential
operators $D_{k+1} := \pi \circ d^{\nabla^{\mathcal{T}_{\mathbb{E}}}} \circ L_k: H_{k}(E)
\rightarrow H_{k+1}(E)$. The properties of the splitting operators together  with  the fact that (\ref{TwisteddeRham}) is a complex imply that,
\begin{equation}\label{BGG}
0 \longrightarrow \Gamma (H_{0}(E))  \overset{ D_{1} }{\longrightarrow} \Gamma( H_{1} (E)) \overset{D_{2}}{\longrightarrow} \ldots \overset{D_{n}}{\longrightarrow} \Gamma( H_{n}(E) ) \longrightarrow 0,
\end{equation}
is a complex. This is called the \emph{BGG complex} associated to $\mathbb{E}$.

Furthermore, the splitting operator induces a chain homotopy between
the differential complexes (\ref{BGG}) and (\ref{TwisteddeRham}). Thus
we have the following:
\begin{thm}[\v{C}SS, CD]\label{BGGCohom}
The cohomology of the complex (\ref{BGG}) is isomorphic to the twisted de Rham cohomology $H^{\bullet}(M, \mathcal{T}_{\mathbb{E}})$.
\end{thm}
\noindent Now using Theorems (\ref{MainIsom}) and (\ref{BGGCohom}), we have the following result.
\begin{thm}\label{BGGIsomGroup}
Let $\rho: SL(n+1,\mathbb{R}) \rightarrow GL(\mathbb{E})$ and $\mu: SO(n,1) \rightarrow GL(\mathbb{F})$ be finite-dimensional representations with $\rho|_{SO(n,1)} \cong \mu$.

If $M \cong \Gamma \backslash \mathbb{H}^{n}$ is a compact hyperbolic manifold, the cohomology groups $H^{\bullet}(\Gamma,\mathbb{F})$ are isomorphic to the cohomology of the BGG complex associated to $\mathbb{E}$.
\end{thm}

\subsection{Examples of BGG complexes}\label{BGGex}

\begin{exam}
For the case where $\mathbb{E}=(\mathbb{R}^{n+1})^{*}$ is the dual of the defining representation of $SL(n+1)$ (cf.\ example \ref{dualex}), the beginning of the corresponding BGG complex begins,
\[
0 \longrightarrow C^{\infty}(M) \overset{D_{0}}{\longrightarrow} \Gamma(S^{2} TM^{*}) \overset{ D_{1} }{\longrightarrow} \Gamma ( \Lambda^{2} TM^{*} \odot TM^{*}) \overset{D_{2}}{\longrightarrow} \cdots.
\]
The operators are defined by, 
\begin{align}\nonumber
(D_{0}s)(X,Y) &= (\nabla_{X}^{g} ds)(Y)- g(X,Y)s, \; s \in C^{\infty}(M) \\ \nonumber
2 (D_{1}u)(X,Y,Z) &=  (\nabla^{g}_{X}u)(Y,Z) - (\nabla^{g}_{Y}u)(X,Z) , \; u \in \Gamma(S^{2} TM^{*}), \nonumber
\end{align}
and there are similar expressions for the higher $D_{k}$, $k \geq 2$.
\end{exam}
In this example, the restriction of $\mathbb{E}$ to $SO(n,1)$ is just the dual of the defining representation of $SO(n,1)$, which we will still write as $(\mathbb{R}^{n+1})^{*}$. 
Using the combination of Theorem \ref{MainIsom} and Theorem \ref{BGGCohom}, the complex in the preceding example yields immediate geometric
interpretations of the group cohomology $H^{\bullet}(\Gamma, (\mathbb{R}^{n+1})^{*})$.

For example, on any manifold with a projective structure a nowhere
zero solution of $D_{0}$ means that $(M,g)$ is projectively
Ricci-flat, meaning there is an affine connection $\nabla\in \pp$ that is is
projectively Ricci-flat. Furthermore any solution of $D_{0}$ is
necessarily non-zero on an open dense subset of $M$
\cite{CapGoverHammerl}, Section 3.2.  Theorem \ref{MainIsom} gives
that $H^{0}(\Gamma, (\mathbb{R}^{n+1})^{*}) \cong H^{0} (M,
\mathcal{T}^{*})$, and so Theorem \ref{BGGCohom} gives an
interpretation of the fact that $H^{0}(\Gamma, (\mathbb{R}^{n+1})^{*})
=0$. Specifically, it means that $D_0$ has no solutions, and so the
Levi-Civita connection on $M$ is not projectively related to a
Ricci-flat (and hence flat) affine connection.

\begin{exam}[This generalises the  previous example]\label{DualBGG}
For the case $\mathbb{E}=S^{k}(\mathbb{R}^{n+1})^*$, 
the corresponding BGG complex begins,
\[
0 \longrightarrow C^{\infty}(M) \overset{D_{0}}{\longrightarrow} \Gamma ( S^{k+1} TM^{*}) \overset{D_{1}}{\longrightarrow} \Gamma( \Lambda^{2} TM^{*} \odot S^{k} TM^{*} ) \overset{D_{2}}{\longrightarrow} \cdots,
\]
If $s \in C^{\infty}(M)$ and $X_{i} \in \Gamma(TM) $ are vector fields, the operator $D_{0}$ is defined by, $D_{0}s :=\mathbf{Sym} (\tilde{D}_{0} s)$, where,
\[
(\tilde{D}_{0}s)(X_{1},X_{2},\ldots,X_{k+1}) :=\nabla^{g}_{X_{1}} \nabla^{g}_{X_{2}} \cdots \nabla^{g}_{X_{k+1}} s - c_k g(X_{1},X_{2}) \nabla^{g}_{X_{3}} \ldots \nabla^{g}_{X_{k+1}} s, 
\]
where $c_k$ is a constant (depending on $k$) and $\mathbf{Sym}$ is the
symmetrisation map from $\otimes^{k+1} TM^{*}$ to $S^{k+1} TM^{*}$.

If $u \in \Gamma(S^{k+1} TM^{*})$, the operator $D_{1}$ is defined by, $D_{1} u:=\mathbf{Proj}(\tilde{D}_{1} u)$, where,
\[
\tilde{D}_{1}u (X_{1},X_{2}, \ldots, X_{k+2}) := (\nabla^{g}_{X_{1}} u)(X_{2}, \ldots, X_{k+2}), 
\] 
and $\mathbf{Proj}$ is the projection map from $TM^{*} \otimes S^{k+1} TM^{*}$ to the
Cartan part $\Lambda^{2} TM^{*} \odot S^{k} TM^{*})$. See
e.g.\ \cite{EastwoodGover} for the formulae arising in this example.
\end{exam}

Another example of the geometric interpretation of group cohomology is the following:
\begin{prop}\label{TracefreeCodazzi}
The existence of a global, trace-free Codazzi tensor in $\Gamma(S^{2} TM^{*})$ implies that $H^{1}(\Gamma, (\mathbb{R}^{n+1})^{*})\neq 0$.
\end{prop}
\begin{proof}
First use Theorem \ref{MainIsom} to get that $H^{1}(\Gamma, (\mathbb{R}^{n+1})^{*}) \cong H^{1}(M, \mathcal{T}^{*})$, and then 
use Theorem \ref{BGGIsomGroup} to reduce the problem to finding a nonzero element of the first BGG cohomology. 
Now, use the fact that a Codazzi tensor $u \in \Gamma (S^{2} TM^{*})$ is a solution of the BGG operator $D_{1}: \Gamma( S^{2} TM^{*} ) \rightarrow \Gamma (\Lambda^{2} TM^{*} \odot TM^{*})$ defined in the preceding example. 

Suppose, with a view to contradiction, that there is a function $s$ with $D_{0} s = u$. Then, from the expression for $D_{0}$ in example \ref{DualBGG}, 
we get that, in indices,
\[
\nabla_{a} \nabla_{b} s - g_{ab}s = u_{ab}. 
\]
Contracting the previous equation gives,
\[
-\Delta^{g} s - n s = \tensor{u}{^{c}_{c}},
\]
where we use the Laplacian $\Delta^{g} = - \nabla^{c} \nabla_{c} = \delta d $. But the hypothesis that $u$ is trace-free means that $-n$ is an eigenvalue of $\Delta^{g}$, which is impossible. 
\end{proof}
This proposition gives a tractor interpretation of the proof in \cite{Lafontaine}. In \cite{Lafontaine}, Lafontaine constructs a trace-free Codazzi tensor, and then shows using different methods that this is related to $H^{1}(\Gamma, \mathbb{R}^{n+1})$. Since $\Gamma$ lies in $SO(n,1)$ it follows that  $H^{1}(\Gamma, \mathbb{R}^{n+1})$ is canonically isomorphic to $H^{1}(\Gamma, (\mathbb{R}^{n+1})^{*})$.

\begin{exam}\label{adj}
For the adjoint representation $\frak{g}=\frak{sl}(n+1)$, the first
two homology bundles are $H_{0}(\frak{g}) = TM$ and
$H_{1}(\frak{g})=S^{2} T^*M$, and the first operator $D_{0}$ in the
BGG complex is the Killing operator,
\[
D_{0} : X \mapsto \mathcal{L}_{X} g,
\]
for $X \in \Gamma (TM)$. Theorem (\ref{MainIsom}) implies that, 
\[
H^{0}(M, \mathcal{T}_{\mathfrak{sl}(n+1)}) \cong H^{0}(\Gamma, \Lambda^{2} \mathbb{R}^{n+1}) \oplus H^{0} (\Gamma, S^{2}_{0} \mathbb{R}^{n+1}).
\]
The Vogan-Zuckerman theorem applied to the two groups on the righthand side means that $H^{0}(M, \mathcal{T}_{\mathfrak{sl}(n+1)})=0$ and so Theorem \ref{BGGCohom} 
recovers the well-known fact that there are no global Killing vectors on a hyperbolic manifold.
\end{exam}
These examples are typical of how BGG complexes will be useful in
studying the groups $H^{k}(\Gamma, \mathbb{E})$ and vice versa.  The
bundles in the BGG complex are tensor bundles, and the operators
acting between them are defined entirely using the Riemannian geometry
of $M$. This provides a nice link between the discrete group
cohomology and the geometry of $M$.
\begin{remk}
Apart from the the simplest representations, `manually' determining the homology bundles by examining the action of the $\partial^{*}$-operator is 
too difficult. However, there is an explicit algorithm, based on results of Kostant, for determining the bundles that occur in any given BGG complex with only 
the starting representation as initial data. 

In this algorithm, 
representations determining induced bundles are labelled by decorated Dynkin diagrams, and one uses the action of the affine Weyl group to determine which representations \textup{(}and in which order\textup{)} occur in the complex. See \cite{BastonEastwood}, especially $4.3$ and $8.5$. We have summarised the results of this algorithm, for the case of interest to this article, in Appendix \ref{Appendix}.

For detailed formulae for the operators occurring in the projective
BGG complex, see \cite{EastwoodGover}.  There is also a uniform method
for determining formulae for a large class of the operators that occur in BGG
complexes, for any parabolic geometry; see
\cite{CalderbankDiemerSoucek} for details.
\end{remk}

\section{Group cohomology and totally geodesic hypersurfaces}\label{MainNonVanishing}

In this section, we will use the isomorphism between the tractor
connection and the (hyperbolic) canonical flat connection to prove a
non-vanishing theorem for the group cohomology of $\Gamma$ with
coefficients in $S^{k}_{0} \mathbb{R}^{n+1}$.  Firstly, we will
examine some of the special properties of the standard tractor bundle
over $M$, and how this relates to totally geodesic hypersurfaces in
$M$.

\subsection{The tractor metric}

The standard tractor bundle on the hyperbolic manifold $M$ has important extra structure; it comes equipped with a signature $(n,1)$ metric, which will be denoted $h$. For $(X \; s)^{t}$ and $(Y \; t)^{t}$ in $\Gamma(\mathcal{T})$, the metric is given by,
\[
h
\big(
\left(
\begin{array}{c}
X \\
s \\
\end{array}
\right),
\left(
\begin{array}{c}
Y \\
t \\
\end{array}
\right)
\big) := g(X,Y) - st.
\]
Using equation (\ref{TracFormula}), the important result that, 
\begin{equation}\label{parallelmetric}
\nabla^{\mathcal{T}} h=0,
\end{equation}
is easily verified, and the metric $h$ will be used to given an isomorphism between $\mathcal{T}$ and its dual.

\begin{remk}
The existence of $h$ is a very special case of a much more general construction developed in \cite[Theorem $3.1$]{GoverMacbeth} \textup{(}see also \cite{Armstrong,CapGoverMacbeth}\textup{)}. 
In these papers, it is shown that the projective equivalence class of an Einstein metric yields a $\nabla^{\mathcal{T}}$-compatible metric on $\mathcal{T}$, and vice-versa.
The fact that the hyperbolic metric is Einstein gives the special case we are considering here.
\end{remk}

\subsection{The normal tractor}\label{normT}

In the next section, we will construct cocycles for $d^{\nabla^{\mathcal{T}}}: \Gamma( S^{k}_{0} \mathcal{T}) \rightarrow \Gamma (TM^{*} \otimes S^{k}_{0} \mathcal{T}) $ using totally geodesic hypersurfaces of $M$.  There
are infinite families of compact, oriented, hyperbolic manifolds with totally
geodesic hypersurfaces: Millson constructs arithmetic examples in
\cite{Millson2}. In Section \ref{nontriv}, the question of whether the
cocycles found are trivial in $H^{1}(\Gamma, S^{k}_{0} \mathbb{R}^{n+1})$ will
be investigated. A key tractor tool that we shall require is the {\em
  normal tractor} field, which is a tractor extension of the normal
vector field along a hypersurface.

Let $\iota: \Sigma \hookrightarrow M$ be a closed, compact, orientable
hypersurface in the hyperbolic manifold $(M,g)$. Let $N \in \Gamma(
TM|_{\Sigma})$ be the unit normal to $\Sigma$, with $g(N,N)=1$ at all
points of $\Sigma$.  The unit normal induces the decomposition of
vector bundles,
\[
TM|_{\Sigma} \cong T \Sigma \oplus N\Sigma,
\]
where the normal bundle $N \Sigma$ is a trivial line bundle and spanned pointwise by $N$.

Now, define the \emph{projective normal tractor} $\nu \in \Gamma(\mathcal{T}|_{\Sigma})$ 
to be,
\[
\nu :=
\left(
\begin{array}{c}
N \\
0 \\
\end{array}
\right).
\]
\begin{lemma}
The normal tractor $\nu$ is parallel along $\Sigma$ with respect to the tractor connection if and only if the surface $\Sigma$ is totally geodesic.
\end{lemma}
\begin{proof}
Differentiate the projective normal tractor along the surface $\Sigma$ using the formula (\ref{TracFormula}) for the tractor connection $\nabla^{\mathcal{T}}$, to get, 
for any $X \in \Gamma( T \Sigma )$,
\[
\nabla_{X} \nu =
\left(
\begin{array}{c}
\nabla_{X} N \\
g(X,N) \\
\end{array}
\right).
\]
The desired result follows because the term $g(X,N)$ is zero for $X \in \Gamma( T \Sigma)$, and $\nabla_{X} N = 0$ if and only if $\Sigma$ is totally geodesic.
\end{proof}

The induced Riemannian metric $\tilde{g}$ on the hypersurface $\iota: \Sigma
\hookrightarrow M$ gives $\Sigma$  an intrinsic submanifold
projective structure and hence an intrinsic tractor connection.
In particular  
if the induced hypersurface metric is a  hyperbolic metric then 
$\Sigma$ has the structure of a hyperbolic manifold   $(\Sigma , \tilde{g} )\cong \Gamma' \backslash \mathbb{H}^{n-1}$. In this case the tractor connection arises  
using the construction given in section \ref{FlatModels}. 

As for the ambient standard tractor bundle, we have the canonical
splitting $\mathcal{T} \Sigma \cong T \Sigma \oplus \mathbb{R}$.  By
identifying $T \Sigma$ with a sub bundle of $TM|_{\Sigma}$ in the
usual way, we can evidently identify $\mathcal{T} \Sigma$ with a sub
bundle of $\mathcal{T}|_{\Sigma}$.  If the hypersurface $\Sigma$ is
totally geodesic, then the induced Riemannian metric on $\Sigma$ is a
hyperbolic metric, and the following Lemma shows that the inclusion of
$\mathcal{T} \Sigma$ into $\mathcal{T}|_{\Sigma}$ has nice properties:
\begin{lemma}\label{OrthogonalSplit}
For a totally geodesic hypersurface $\Sigma$, there is an orthogonal splitting (for the tractor metric $h$),
\[
\mathcal{T}|_{\Sigma} \cong \mathcal{T} \Sigma \oplus \mathcal{N} \Sigma,
\]
where $\mathcal{N}\Sigma$ is a trivial line bundle spanned pointwise by $\nu$.
The restriction of the ambient tractor connection $\nabla^{\mathcal{T}}$ to 
$\mathcal{T}|_{\Sigma}$ is a direct sum of $\nabla^{\mathcal{T} \Sigma}$ and the trivial connection on the line bundle $\mathcal{N} \Sigma$. 
\end{lemma}
\begin{proof}
Along the hypersurface $\Sigma$, for $Y \in \Gamma(TM)$ and $s \in C^{\infty}(M)$, a section $(Y \;  s)^{t} \in \Gamma(\mathcal{T}|_{\Sigma})$ splits,
\[
\left(
\begin{array}{c}
Y \\
s \\
\end{array}
\right) = 
\left(
\begin{array}{c}
Y^{\top}\\
s
\end{array}
\right) + 
\left(
\begin{array}{c}
Y^{\perp} \\
0 
\end{array}
\right),
\]
where $Y^{\top}$ and $Y^{\perp}$ are the tangential and normal components of $Y$ respectively. 
Sections in the first summand are sections of $\mathcal{T} \Sigma$, using the identification above. Direct calculation using the metric $h$ show that this splitting is orthogonal.

Recall that $\tilde{g}$ be the induced metric on $\Sigma$. If $X \in
\Gamma(T \Sigma)$ is a tangential vector field, then the formula for
the ambient tractor connection gives that,
\begin{align}\nonumber
\nabla^{\mathcal{T}}_{X}
\left(
\begin{array}{c}
Y \\
s \\
\end{array}
\right) &= 
\left(
\begin{array}{c}
\nabla^{g}_{X} Y + s X \\
X(s) + g(X,Y) \\
\end{array}
\right)
=
\left(
\begin{array}{c}
\nabla^{g}_{X} Y^{\top} + \nabla^{g}_{X} Y^{\perp} + s X \\
X(s) + g(X, Y^{\top}) + g(X, Y^{\perp}) \\
\end{array}
\right) \\ \nonumber
&= 
\left(
\begin{array}{c}
\nabla^{\tilde{g}}_{X} Y^{\top} + sX \\ 
X(s) + \tilde{g}(X, Y^{\top}) \\
\end{array}
\right) 
+
\left(
\begin{array}{c}
\nabla_{X}^{g} Y^{\perp} \\
0 \\ 
\end{array}
\right), \\ \nonumber
\end{align}
where the Gauss formula, and the fact that $\Sigma$ is totally
geodesic were used to give that $\nabla^{g}$ can be replaced with
$\nabla^{\tilde{g}}$. The expression in the first summand of the
second line of the preceding equation is the formula for
$\nabla^{\mathcal{T} \Sigma}$, according to equation \ref{TracFormula}
\end{proof}

\begin{lemma}\label{NormalExt} Let $U$ be a tubular neighbourhood of a 
compact, totally geodesic hypersurface $\Sigma$.
 The normal tractor of $\Sigma$ can be extended to a section
 $\nu\in \Gamma (\mathcal{T}|_U)$ that is parallel.
\end{lemma}
\begin{proof}
This is elementary, since the tractor connection $\nabla^{\mathcal{T}}$ is flat. 
\end{proof}
The section $\nu \in \Gamma ( \mathcal{T}|_{U})$ constructed in the preceding Lemma can be extended to a smooth section of $\mathcal{T}\to M$ with support containing $U$, and we will henceforth use the same symbol $\nu$ for this section.

\subsection{Constructing tractor cocycles}\label{constcocyv}
For any smooth compact hypersurface $\iota: \Sigma \hookrightarrow M
$, recall that the {\em Poincar\'{e} dual} of $\Sigma$ is the
cohomology class in $H^{1}(M, \mathbb{R})$ determined by the property
that, for any one-form $\omega_{\Sigma}$ representing this class in de
Rham cohomology,
\[
\int_{\Sigma} \iota^{*} \eta = \int_{M} \eta \wedge \omega_{\Sigma},
\]
for all closed $(n-1)$-forms $\eta$ on $M$.

Additionally \cite{BottTu}, it possible to choose a representative one-form for the Poincar\'{e} dual whose support can be shrunk to any given tubular neighbourhood of $\Sigma$.
These properties, and Lemma (\ref{NormalExt}) give that:

\begin{lemma}\label{keyL}
Define $\tau:= (\nu \otimes \nu \otimes \ldots \nu)_{0}$ to be the trace-free part of, 
\[
\underbrace{( \nu \otimes \nu \otimes \ldots \otimes \nu)}_{k-times}.
\]
If $\Sigma$ is totally geodesic, and $\omega_{\Sigma}$ is a Poincare dual of $\Sigma$ with support contained in the region where $\nu$ is parallel, then the tractor valued one-form $\omega_{\Sigma} \otimes \tau \in \Gamma( TM^{*} \otimes \mathcal{T})$ defines a cocycle for $d^{\nabla^{\mathcal{T}}}$ and hence, using theorem \ref{MainIsom}, an element of $H^{1}(\Gamma , S^{k}_{0}\mathbb{R}^{n+1})$.
\end{lemma}
The question that we will next address is whether the element of $H^{1}(\Gamma, S^{k}_{0} \mathbb{R}^{n+1})$ defined by $\omega_{\Sigma} \otimes \tau$ 
is trivial or not. Notice that the relationship to the
topology of $\Sigma$ is not straightforward, in the sense that even if $[\Sigma]=0 \in H_{n-1}(M, \mathbb{R})$, it does not directly follow that the cocycle defined by $\omega_{\Sigma} \otimes \tau$ is trivial in cohomology.
\subsection{Non-triviality}\label{nontriv}

In this section we will prove that the cocycle constructed in Lemma \ref{keyL} associated to a
totally geodesic, compact hypersurface is nontrivial in cohomology.
\begin{thm}\label{MainTheorem}
If a compact hyperbolic manifold $M$ contains an orientable, compact, totally geodesic hypersurface $\Sigma$, then $H^{1}( \Gamma, S^{k}_{0} \mathbb{R}^{n+1}) \neq 0$, for $k \geq 1$.
\end{thm}
\begin{proof}
Recall that the normal tractor $\nu$ constructed in Lemma \ref{NormalExt} is parallel in a tubular neighbourhood $U$ of $\Sigma$. Fix a diffeomorphism that identifies $U$ with $\Sigma \times \mathbb{R}$, and use the notation $\pi$ for the projection onto the first factor. Choose a Poincar\'{e} dual form $\omega_{\Sigma}$ so that the 
the support of $\omega_{\Sigma}$ is the open set, 
\[
U' := \{ (x,t) \in U \cong \Sigma \times \mathbb{R} \; | \; -1 < t < 1 \}.
\] 
(This is possible, see Proposition $6.25$ of \cite{BottTu}).
Define $\tau:= (\nu \otimes \nu \otimes \ldots \nu)_{0}$ as above. The section $\tau$ is a trace-free element of $\Gamma(S^{k} \mathcal{T})$ which is parallel on $U$.

Consider the cocycle $\omega_{\Sigma} \otimes \tau$, and the corresponding element $[\omega_{\Sigma} \otimes \tau ] \in H^{1}(M, S^{k}_{0} \mathcal{T})$. 
Theorem \ref{MainIsom} shows that to prove Theorem \ref{MainTheorem} it is sufficient to show that $[ \omega_{\Sigma} \otimes \tau]$ is nontrivial in $H^{1}(M, S^{k}_{0} \mathcal{T})$.
This will be achieved by showing that there is no section $\sigma \in \Gamma(S_{0}^{k} T)$ with $\nabla^{\mathcal{T}} \sigma = \omega_{\Sigma} \otimes \tau$. 

Assume, with a view to contradiction, that such a section $\sigma$ 
exists. 
Now, for a small positive $\epsilon \in \mathbb{R}$, define the two subsets, 
\[
V_{+} := \{ (x,t) \in U \; | \; t > - \epsilon \} \cup (M \setminus U), \; V_{-} := \{ (x,t) \in U \; | \: t < \epsilon \} \cup (M \setminus U).
\]
Since $\nabla^{\mathcal{T}} \sigma = \omega_{\Sigma} \otimes \tau$, it follows that $\sigma$ is parallel over $M \setminus U'$. It is not necessarily parallel on the sets $V_{+}$ and $V_{-}$.
Our strategy will be to use parallel transport along the fibres of $\Sigma \times \mathbb{R}$ from $M \setminus U'$ to $V_{+}$ and $V_{-}$, in order to extend $\sigma|_{M \setminus U'}$ into parallel sections $\sigma_{\pm}$ over $V_{\pm}$. This will then be shown to lead to a contradiction. 

\begin{enumerate}
\item[\underline{Step 1}:]
To implement this construction, we first need to show that $\sigma|_{M \setminus U'}$ is not trivial. 
To do this, use that, for any $X \in \Gamma(TM|_{U})$,
\[
X.h(\sigma, \tau) = h(\nabla_{X}^{\mathcal{T}} \sigma, \tau) = h(\tau,\tau) \omega_{\Sigma}(X). 
\]
Now recall that $\nu$ is length $1$, so that $h(\tau,\tau) = c_{n,k}$ is constant, and so $\omega_{\Sigma}$ is exact on this region. For some $x=(x,0) \in \Sigma \times \mathbb{R}$, integrate $\omega_{\Sigma}$ along the line $l(t):= (x, -1-\epsilon +(1+\epsilon)2t)$, and use the 
fact \cite[1.6]{BottTu} that the integral of $\omega_{\Sigma}$ along the fibre of $\Sigma \times \mathbb{R}$ is equal to one to get,
\begin{equation}\label{sigmanotzero}
1= \int_{l} \omega_{\Sigma} = \frac{1}{c_{n,k} }   \left( h(\sigma, \tau)_{(x,1+\epsilon)} -  h(\sigma, \tau)_{(x,-1-\epsilon)} \right).
\end{equation}
This implies that $\sigma$ is not identically zero in the complement of $U'$. In fact, since $\sigma$ is parallel outside of $U'$ this shows that $\sigma$ never equals $0$ on $M \setminus U'$. 

Now we will use $\sigma$ to define sections on $V_{\pm}$.
\item[\underline{Step 2}:]
This is done as follows: choose a point $p=(x,1 + \epsilon) \in U \setminus U'$ and a simply connected open set $W \subset \Sigma$, with $x \in W$. 
We can define a section, $\sigma_{+}$ on the set $\{(w,t) \in U \; | \; w \in W, t > -\epsilon \}$ by parallel transport of the vector $\sigma(p) \in \mathcal{T}_{p}$ along any path 
in $\{ (w,t) \in U \; | \; w \in W, t > -\epsilon \}$. Since the tractor connection $\nabla^{\mathcal{T}}$ is flat, the precise path is irrelevant, since parallel transport only depends on the homotopy class of a path. 

For another simply connected set $W' \subset \Sigma$ with $x' \in W'$, and $p'=(x', 1+ \epsilon) \in U \setminus U'$, we can similarly define a section $\sigma'_{+}$ on $\{ (w',t) \in U \; | \; w' \in W', \; t > -\epsilon \}$. 
The sections $\sigma_{+}$ and $\sigma'_{+}$ thus defined will \emph{agree} on $\{ (w,t) \in U \; | \; w \in W \cap W', \; t > - \epsilon \}$. 

To see why, notice that, for any $w \in W \cap W'$, there is a path that joins $p$ with $w$ and passes through $(\pi(w),1+\epsilon)$. 
Specifically, if $\gamma(s):[0,1] \mapsto \Sigma$ is a path joining $x$ with $\pi (w)$, then $(\gamma(s),1+\epsilon)$, concatenated with translation along the fibre $\mathbb{R}_{\pi (w)}$ is the required path. 
Likewise, there is a similar path joining $p'$ with $w$, which passes through $(\pi(w),1+\epsilon)$. Since $\sigma$ is parallel outside of $U'$ and the tractor connection $\nabla^{\mathcal{T}}$ is flat, these paths show that $\sigma_{+}(w)=\sigma'_{+}(w)$, because they are both equal to the parallel transport of $\sigma_{(\pi(w),1+\epsilon)}$ down the fibre $\mathbb{R}_{\pi(w)}$.
Thus, $\sigma_{+}=\sigma'_{+}$ on $W \cap W'$. 
Using a covering $\{ W_{i} \}$ of $\Sigma$ by simply-connected open sets, continue in this way to define a section $\sigma_{+}$ on $\{(x,t) \in U \; | \; t > -\epsilon \}$.

Similarly, we can define a parallel section $\sigma_{-}$ on $\{ (x,t) \in U \; | \; t < \epsilon\}$, by starting at a point $q=(x,-1-\epsilon)$ and using parallel transport of $\sigma(q)$ to define a section $\sigma_{-}$ on 
$\{ (w,t) \; | \; w \in W, t < \epsilon \}$. Then, continue as before, using a covering $\{W_{i} \}$ of $\Sigma$ to define a parallel section $\sigma_{-}$ on $\{ (x,t) \in U \; | t < \epsilon \}$.

By declaring $\sigma_{\pm}(p):=\sigma(p)$ at all points $p\in M
\setminus U$, we get two smooth parallel sections on $V_{\pm}$, since
by construction $\sigma_{\pm}$ agrees with $\sigma$ on $U \setminus
U'$.


\item[\underline{Step 3}:] 
Now, we will show that $\sigma_{+} = c^{+} \tau$, for some constant $c^{+}$. 
To see this, notice that Lemma \ref{OrthogonalSplit} shows that there is an orthogonal splitting of $S^{k}_{0} \mathcal{T}|_{\Sigma}$ into a direct sum; with summands either of the form $S^{l}_{0} \mathcal{T} \Sigma$, with $1 \leq l \leq k$, 
or of the form,
\[
\underbrace{(\mathcal{N} \Sigma \otimes \mathcal{N} \Sigma \otimes \ldots \otimes \mathcal{N} \Sigma)_{0}}_{k \; times}.
\]
Moreover, with the induced Riemannian metric, $\Sigma$ is itself a compact hyperbolic manifold $\Gamma' \backslash \mathbb{H}^{n-1}$. Theorem \ref{MainIsom} therefore shows that the orthogonal projection of the parallel section 
$\sigma^{+}$ onto any of the summands $S^{l}_{0} \mathcal{T} \Sigma$ gives an element of $H^{0} (\Gamma', S^{l}_{0} \mathbb{R}^{n+1})$. But this must be zero, because the case of the Vogan-Zuckerman theorem discussed in equation \ref{AllVanish}, section \ref{VoganZuckerman} (and also remark \ref{thomasD}). So $\sigma_{+}$ has no component in $S^{l}_{0} \mathcal{T} \Sigma$, which means that $\sigma_{+}=c^{+} \tau$, for some constant $c_{+}$. The constant $c_{+}$ must be nonzero, because $\sigma$ is nowhere zero on $M \setminus U'$ (which we showed above using Equation \ref{sigmanotzero}), and $\sigma_{+}$ agrees with $\sigma$ outside of $U'$.

The same reasoning applied to $\sigma_{-}$ means that $\sigma_{-}=c^{-} \tau$ for some (necessarily non-zero) constant $c^{-}$. 

\item[\underline{Step 4}:] 
From the fact (Equation \ref{parallelmetric}) that $\nabla^{\mathcal{T}} h=0$, we get that $c^{+} = \pm c^{-}$. 
This means that $\sigma_{+} \otimes \sigma_{+}$ and $\sigma_{-} \otimes \sigma_{-}$ are parallel sections of $\otimes^{2k} \mathcal{T}|_{V_{+}}$ and $\otimes^{2k} \mathcal{T}|_{V_{-}}$, respectively, which agree on the overlap $V_{+} \cap V_{-}$.  This means that there is a gobally parallel section $\varpi \in \Gamma( \otimes^{2k} \mathcal{T})$ such that $\varpi=\sigma_{+} \otimes \sigma_{+}$ on $V_{+}$ and $\varpi=\sigma_{-} \otimes \sigma_{-}$ on $V_{-}$.

Since $\varpi$ locally agrees with $\sigma_{+} \otimes \sigma_{+}$ or $\sigma_{-} \otimes \sigma_{-}$, the trace-free, symmetric part of $\varpi$ is not zero. But then Theorem 
\ref{MainIsom} means that there is a nonzero element of $H^{0} (\Gamma, S^{2k}_{0} \mathbb{R}^{n+1})$, which contradicts the case of the Vogan-Zuckerman theorem in equation \ref{AllVanish}, section \ref{VoganZuckerman}.
\end{enumerate}  
\end{proof}
\newpage

\appendix
\section{General formulae for bundles occuring in the BGG complex}\label{Appendix}
To describe the end result of the algorithmic method of determining BGG complexes, let's fix some notation from representation theory. Fix a Cartan subalgebra $\mathfrak{h} \leq \mathfrak{g}$, and let $\{ \alpha_{i} \}$ be a set of simple roots for $( \mathfrak{g}, \mathfrak{h})$. Let $\alpha_{i}^{\vee}$ be a corresponding set of co-roots, so that $\alpha_{i}^{\vee} = 2 \alpha_{i} / (\alpha_{i}, \alpha_{i})$, and $\varpi_{i}$ be the corresponding fundamental weights, with $(\varpi_{i}, \alpha_{j}^{\vee}) = \delta_{ij}$.

We will use the correspondence between $SL(n+1,\mathbb{R})$-representations and highest weight vectors in $\mathfrak{h}^{*}$, which will be written in the basis $\{ \varpi_{i} \}_{i=1}^{n} \in \mathfrak{h}^{*}$. Furthermore, the vector in $\mathfrak{h}^{*}$ labelling a representation will also denote the corresponding associated bundle to $\mathcal{G}_{\Gamma}$. With these conventions, the BGG complex on $\Gamma \backslash \mathbb{H}^{n}$ has the following bundles (in this order), 
\begin{align} \nonumber
(a_{1}, a_{2}, \ldots, a_{n}) &\to (a_{1}-2, a_{1} +a_{2} +1, a_{3}, \ldots, a_{n}) \\ \nonumber
&\to (a_{1}-a_{2}-3, a_{1}, a_{2} + a_{3} +1, a_{4}, \ldots, a_{n}) \\ \nonumber
&\to (a_{1}-a_{2}-a_{3}-4,a_{1},a_{2},a_{3}+a_{4}+1,a_{5}, \ldots, a_{n}) \to  \cdots \\ \nonumber
&\to (-a_{1}-a_{2}-\cdots - a_{n-1} -n, a_{1},a_{2}, \ldots , a_{n-2}, a_{n-1} + a_{n} +1)   \\ \nonumber
&\to (-a_{1} -a_{2} - \ldots - a_{n} - (n+1),a_{1}, a_{2}, \ldots, a_{n-1}). \\ \nonumber
\end{align}

\end{document}